\documentclass[11pt]{article}
\usepackage{amsfonts,amssymb,amsmath,amsthm,amscd}
\usepackage[margin=1.1in]{geometry}
\newtheorem{proposition}{Proposition}[section]

\newtheorem{theorem}[proposition]{Theorem}

\newtheorem{remark}[proposition]{Remark}

\begin{document}

\newcommand{\EC}{{\sf EC}}
\newcommand{\AS}{{\sf AS}}
\newcommand{\DD}{{\sf D}}

\title{On Abelian subvarieties of bounded degree \\ in a polarized Abelian variety}
\date{}
\author{Lucio Guerra}
\maketitle

\begin{abstract}
If $A$ is an Abelian variety, endowed with a polarization $L$, we study the function $N_A(t)$ which counts the number of Abelian subvarieties $S$ in $A$ such that for the induced polarization $L|_S$ the Euler characteristic  $\chi(L|_S)$ is bounded above by $t$. We give an estimate for the asymptotic order of growth of this function.
\end{abstract}

\section{Introduction}

Let $A$ be an Abelian variety, endowed with a polarization $L$. (We view polarizations as divisor classes, rather than line bundles.) Let $S$ be a subvariety of $A$, of dimension $n$. The intersection number $$S \cdot L^n = (L|_S)^n$$ can be considered as a geometric degree of $S$, at least when $L$ is a very ample divisor class. If $S$ is an Abelian subvariety of $A$, the geometric form of the Riemann Roch theorem for Abelian varieties says that
$$\frac{(L|_S)^n}{n!} = \chi(L|_S)$$
is the Euler characteristic of the line bundle associated to the restricted polarization.
So, for an Abelian subvariety, the Euler characteristic $\chi(L|_S)$ can be considered as a `reduced degree'. Denote with the symbol $\AS_A$ the collection of all Abelian subvarieties of $A$. For every positive integer or real number $t$, define $$N_A(t) := \# \big\{ S \in \AS_A \mbox{ s.t. } \chi(L|_S) \leq t \big\} \smallskip.$$

That this is a finite number was already known to Poincar\'e, at least for the subcollection of subvarieties of dimension $n=1$. The general statement follows from a result of Birkenhake and Lange \cite{BL}, to the effect that the collection of all Abelian subvarieties with bounded exponent in $A$ is finite. In a couple of papers \cite{Gu1},\cite{Gu2} we studied the problem of giving an asymptotic estimate for the finite number of elliptic curves with bounded degree in $A$. The aim of the present work is to show that the method of the previous papers can be adjusted to the general situation. The argument leads to make explicit the relation between the rate of growth of the counting function $N_A(t)$ and the complete decomposition of $A$ up to isogeny into a product of simple Abelian varieties.

Assume that the complete decomposition of $A$ is given by an isogeny
$$A \approx B_1^{k_1} \times \cdots \times B_q^{k_q}$$
where $B_1,\ldots,B_q$ are simple Abelian varieties, each other non-isogenous, 
and where $k_1,\ldots,k_q$ are positive integers. Define the integers: $k$, the maximum among the multiplicities $k_j$, and $h$, the maximum rank of an endomorphism group $End(B_j)$.
In terms of these data, we prove the following
\begin{theorem}
\label{T}
There is an asymptotic estimate
$$N_A(t) = O\left(t^{\,q (kh+2)(k-1)}\right).$$
\end{theorem}

The main tools for the argument are: a thorough use of the complete reducibility theorem in the relative version involving a given polarization (cf. \S \ref{reducibility}, especially Proposition \ref{deg-chi}), and a classical result in Number Theory concerning the number of lattice points in a bounded region of Euclidean space (cf. \S \ref{lattice points}). The proof of Theorem \ref{T} is given in the final \S \ref{finale}.

\section{Some remarks about reducibility}
\label{reducibility}

\subsection{Reducibility with respect to a polarization}
\label{RP}
Let $A$ be an Abelian variety, with a polarization $L$. Let $S$ be an Abelian subvariety of $A$. There is in $A$ an Abelian subvariety $S'$ such that 
\begin{itemize}
\item
[$-$] $S \cap S'$ is a finite set and $S + S' = A$, so that 
the sum map $$s: S \times S' \rightarrow A$$ is an isogeny, and
\item
[$-$] the pullback polarization from $A$ splits as  
$$s^\ast (L) = p^\ast (L|_S) + (p')^\ast (L|_{S'})$$
\end{itemize} 
where $p,p'$ denote the natural projections of  the product $S \times S'$ to the factors $S,S'$, respectively.
The complementary Abelian subvariety $S'$ is uniquely determined. 

\subsection{On the degree of the sum isogeny}
\label{remark}
In the setting of the preceding section, the following holds:
$$\deg(s)\; \chi(L) = \chi(L|_S)\; \chi(L|_{S'}).$$
In fact, one has $\chi(s^\ast(L)) = \deg(s)\, \chi(L)$ (the behavior of the Euler characteristic under an etale finite morphism) and on the other hand $\chi(s^\ast(L)) = \chi(p^\ast (L|_S) + (p')^\ast (L|_{S'})) = \chi(L|_S)\, \chi(L|_{S'})$ (by the K\"unneth formula for the cohomology of shaves). Furthermore, concerning the degree $\deg(s) = \#(S \cap S')$, we point out the following:

\begin{proposition}
\label{deg-chi}
For the degree of the sum isogeny in the reduciblity theorem, one has the upper bound
$$\deg(s) \leq \chi(L|_S)^2.$$
\end{proposition}

\begin{proof}
When $L$ is a principal polarization on $A$, the relation between the types of the induced polarizations $L|_S$ and $L|_{S'}$ is completely understood (see \cite{CAV}, 12.1.5, p. 366). If the type for the smaller dimension is $(d_1,\ldots,d_m)$ then for the higher dimension the type is $(1,\ldots,1,d_1,\ldots,d_m)$. It follows that $\chi(L|_S) = \chi(L|_{S'})$, both equal to the product $d_1 \cdots d_m$, and therefore $\deg(s) = \chi(L|_S)^2$ follows from the equation above, with $\chi(L) = 1$.

If $L$ is an arbitrary polarization, there is an isogeny $g: A \rightarrow B$ onto an Abelian variety $B$ having a principal polarization $M$ such that $L = g^\ast(M)$ (see \cite{CAV}, 4.1.2, p. 71). 
If $S$ is an Abelian subvariety of $A$, define $T := g(A)$ and let $d_S$ be the degree of the restriction $g_S: S \rightarrow T$. 
One has $\chi(L|_S) = d_S\, \chi(M|_T)$.

Let $T'$ be the complement of $T$ in $B$ with respect to $M$. Then $S'$, the complement of $S$ in $A$ with respect to $L$, coincides with $g^{-1}_{\, 0}(T')$, the connected component of $0$ in the pre-image $g^{-1}(T')$. In order to check this, it is enough to work with the commutative diagram 
$$\begin{CD}
S \times g^{-1}_{\, 0}(T') @>{s}>> A \\
@V{\overline{g \times g}}VV @VV{g}V \\
T \times T' @>>{t}> B
\end{CD}$$
where $s,t$ are the sum homomorphisms,
and check that the pullback of $M$ 
splits on $S \times g^{-1}_{\, 0}(T')$.
We have just seen that
$\#(T \cap T') = \chi(M|_T)^2$ holds since $M$ is a principal polarization.
Then \smallskip 
$\#(S \cap S') \leq \#(S \cap g^{-1}(T')) \leq d_S\, \#(T \cap T')
\leq {d_S}^2 \chi(M|_T)^2 = \chi(L|_S)^2$.
\end{proof}

Finally, here is an immediate consequence:
$$\chi(L|_{S'}) \leq \chi(L)\, \chi(L|_S).$$

\subsection{Complete reducibility with respect to a polarization}
\label{CRP}
As a consequence of the reducibility theorem quoted in \S \ref{RP} one has the following theorem of complete reducibility into simple factors.

Let $A$ be an arbitrary Abelian variety, endowed with a polarization $L$. There exist:
\begin{itemize}
\item[$-$]
a sequence of simple Abelian varieties $B_1,\ldots,B_q$, each other non-isogenous, 
\item[$-$] a corresponding sequence of positive integers $k_1,\ldots,k_q$, 
\item[$-$] 
and an isogeny $$s: B_1^{k_1} \times \cdots \times B_q^{k_q} \rightarrow A$$such that the pullback polarization $s^\ast(L)$ splits. \smallskip
\end{itemize}
In general we say that a polarization on a product variety $S_1 \times \cdots \times S_k$ is a split polarization if it is of the form
$$p_1^\ast(L_1) + \cdots + p_k^\ast(L_k)$$
where $L_i$ is a polarization on $S_i$ and $p_i$ is the projection to the $i$th factor.

In the theorem of complete reducibility,
the collection of pairs $(B_j,k_j)$ is determined from $A$ alone, up to isogeny, however the isogeny $s$ that is described above also depends on the polarization $L$.

\subsection{Behavior under isogenies}
\label{bui}

Let $A$ and $B$ be Abelian varieties, and let 
$g : B \rightarrow A$ be an isogeny, of degree $d$.
Let $L$ be a polarization on $A$ and consider on $B$ the pullback polarization $g^\ast(L)$. 
There is a one to one correspondence 
$\AS_A \overset{_\sim}{\longrightarrow} \AS_B.$
Given $S \subset A$ the corresponding $S^\ast$ in $B$ is
the connected component of $0$ in the pre-image $\varphi^{-1}(S)$. 
The restricted isogeny $S^\ast \rightarrow S$ has degree $d_S \leq d$
(in fact a divisor of $d$). With respect to the polarizations one has
$\chi((g^\ast L)|_{S^\ast}) = d_S\, \chi(L|_S)$.
So there is the chain of inequalities
$\chi(L|_S) \leq \chi((g^\ast L)|_{S^\ast}) \leq d\, \chi(L|_S)$.
It follows that the functions counting Abelian subvarieties in $A$ and in $B$
are related by the inequalities
$$N_A(t) \leq N_B(dt) {\rm \ \ \ and \ \ \ } N_B(t) \leq  N_A(t)$$
and therefore one has
$$N_A(t) = O(t^e) \mbox{ \; if and only if \; } N_B(t) = O(t^e).$$

\begin{remark} \label{R} \em
Because of the invariance under isogeny of the asymptotic estimate for the counting function, and because of the complete reducibility theorem with respect to a polarization, in order to prove Theorem \ref{T} it is enough to confine oneself to the case of a product of simple Abelian varieties endowed with a split polarization. This is what we do in the following.
\end{remark}

\section{Bounding homomorphisms}

\subsection{A quadratic form}
Let $A,A'$ be Abelian varities, of dimensions $n,n'$, endowed with polarizations $\Theta,\Theta'$. Consider the function
$$Hom(A,A') \longrightarrow \mathbb Z$$
which is defined by the assignment
$$f \longmapsto \dfrac{\Theta^{n-1}}{(n-1)!} \cdot f^\ast\Theta'.$$ 
It is known that this is
a quadratic form,
always non-negative.
positive definite if $A$ is simple.
The second and the third assertion are almost apparent, by some intersection theory.
The first assertion for the special case $A = A'$, i.e. for an endomorphism group $End(A)$, can be found for instance in \cite{CAV}, \S 5.1, p. 117; and the general case can be reduced to this special one by embedding the homomorphism group into a suitable endomorphism group in such a way that the function defined above for homomorphisms can be recovered, up to some scalar factor, as the restriction of the analogous function defined for endomorphisms. 

\subsection{Homomorphisms in a hyperellipsoid}
Assume now that $A$ is a simple Abelian variety.
Define 
$$\Phi(t) := \# \left\{ f \in Hom(A,A') \mbox{ s.t. } \dfrac{\Theta^{n-1}}{(n-1)!} \cdot f^\ast \Theta' \leq t^2 \right\}.$$  
Let $H$ be the rank of the homomorphism group $Hom(A,A')$.

\begin{proposition} 
\label{ellipsoid}
There is an asymptotic estimate
$$\Phi(t) = O(t^{H}).$$
\end{proposition} 

\begin{proof} 
Choose an isomorphism $Hom(A,A') \cong \mathbb Z^H$. The quadratic form on $Hom(A,A')$ corresponds to a quadratic form $Q: \mathbb Z^H \rightarrow \mathbb Z$, and this is naturally extended to $\widehat Q: \mathbb R^H \rightarrow \mathbb R$. As the original form is positive definite, then $Q$ is positive definite and so is $\widehat Q$ too. And so we have
$$\Phi(t) = \# \{ x \in \mathbb Z^H \mbox{ s.t. } Q(x) \leq t^2 \}.$$
Now observe that
$\{ x \in \mathbb Z^H \mbox{ s.t. } Q(x) \leq t^2 \}$
can be written as
$$t \cdot \{ y \in \mathbb R^H \mbox{ s.t. } 
\widehat Q(y) \leq 1 \} 
\cap\, \mathbb Z^H$$ 
and that
$\{ y \in \mathbb R^H \mbox{ s.t. } 
\widehat Q(y) \leq 1 \}$, 
a hyperellipsoid, is compact and convex.
So we are lead to a classical topic in Number Theory, the problem of estimating the number of lattice points belonging to a bounded region in the Euclidean space.
Thus the conclusion follows from the general result quoted below.
\end{proof}

\subsection{Lattice points in a bounded region}
\label{lattice points}
Let $K$ be a subset of $\mathbb R^n$, compact convex and with Lipschitz frontier. The function $\# (tK \cap \mathbb Z^n)$ counts the number of points of the lattice $\mathbb Z^n$ which lie in a deformed region $tK$. There is an asymptotic estimate in the form
$\# (tK \cap \mathbb Z^n) = {\rm Vol}(K)\, t^n + O(t^{n-1})$ (see \cite{M}, Ch. 6, Lemma 2, p. 165),
that we only use in the coarse version
$$\# (tK \cap \mathbb Z^n) = O(t^{n}).$$

\section{Subvarieties from homomorphisms}

Let $A$ and $A'$ be Abelian varieties, and consider the product $A \times A'$ with the natural projections $p,p'$. Let $n$ be the dimension of $A$. If $S$ is an Abelian subvariety in $A \times A'$, we denote by $p^{}_S, p'_S$ the two projections restricted to $S$. 

\subsection{Parametrization}
\label{parametrization1}
We describe the Abelian subvarieties $S$ such that the projection
$p_S: S \rightarrow A$ is surjective and finite, i.e. an isogeny. Define
$$\AS_{A \times A' / \,p} := 
\Big\{ S \in \AS_{A \times A'}
\mbox{ s.t. } p{}_S: S \rightarrow A \mbox{ is an isogeny}
\Big\}.$$

So assume that $p_S: S \rightarrow A$ is an isogeny.
Take the inverse, or dual, isogeny $\widehat {p_S} : A \rightarrow S$ and
define the homomorphism $f:A \rightarrow A'$ as $f := {p'_S} \circ \widehat {p^{}_S}$.
Let $a := \deg(p_S)$. Then $\deg(\widehat {p_S}) = a^{2n-1}$ 
(follows from $p_S \circ \widehat {p_S} = a_A$).
If $\widehat {p_S}$ is viewed as a map $A \rightarrow A \times A'$, 
it is given by $x \mapsto (ax,f(x))$. So we can describe
$S$ as the image of some homomorphism 
$A \rightarrow A \times A'$ of the form $(a,f)$.

For every integer $a > 0$ and every homomorphism $f: A \rightarrow A'$
we denote by $C_{a,f}$ the image of the homomorphism 
$A \rightarrow A \times A'$ given by the pair $(a,f)$, and denote by the same symbol $(a,f)$ the induced homomorphism $A \rightarrow C_{a,f}$.
The condition that $$\deg(a,f) = a^{2n-1}$$ is necessary and sufficient for
the projection $C_{a,f} \rightarrow A$ to be of degree $a$, and
in this case the inverse isogeny $A \rightarrow C_{a,f} \hookrightarrow A \times A'$ is naturally identified to $(a,f)$. 
So we have a correspondence
$$\Bigg\{ \mbox{\begin{tabular}{c} $(a,f)
\in \mathbb Z_{>0} \times Hom(A,A')$ \\
s.t. $\deg(a,f) = a^{2n-1}$ \end{tabular}} \Bigg\}  \longrightarrow
\AS_{A \times A' / \, p}\,.$$
This is a bijective correspondence.
The surjectivity is clear from the discussion above. 
Assume that $(a,f)$ and $(a',f')$ determine the same subvariety $C_{a,f} = C_{a',f'}$. 
(Taking the degree of the first projection we already have $a=a'$.)
Taking the inverse of the first projection we have $(a,f) = (a',f')$.

\subsection{Degree}
\label{degree1}
Assume that on $A \times A'$ is given a split polarization $L = {p}^\ast\Theta +{p'}{}^\ast\Theta'$. 
For a subvariety $S = C_{a,f}$ 
we now study the degree $S \cdot L^n = (L|_S)^n$, for which we give some estimate from below. 

Pulling back to $A$, by means of $\widehat{p _ S} : A \rightarrow S$, we obtain $\widehat{p _ S}^\ast(L|_S)^n = a^{2n-1} (L|_S)^n$. On the other hand, since $\widehat{p _ S}^\ast(L|_S) = a^{\ast}\Theta + f^{\ast}\Theta'$, we have $\widehat{p _ S}^\ast(L|_S)^n = (a^{2}\Theta + f^{\ast}\Theta')^n$. Here we replaced $a^\ast \Theta $ with $a^2\Theta$, in the computation of the intersection number, because $a^{\ast}(\Theta) \equiv \frac{a^{2}+a}{2}\, \Theta + \frac{a^{2}-a}{2}\, (-1)^\ast \Theta$ in $Pic(A)$ (see \cite{CAV}, Proposition 2.3.5, p. 33) and because $\Theta \equiv (-1)^\ast \Theta$ in $H_{2n-2}(A,\mathbb Z)$. It follows that
$$a^{2n-1} (L|_S)^n = (a^{2}\Theta + f^{\ast}\Theta')^n = \sum_{i=0}^{n} \binom{n}{i} a^{2i}\, (\Theta^i \cdot f^\ast {\Theta'}^{n-i}).$$
 \medskip
Furthermore, due to the ampleness of $\Theta$ and $\Theta'$, one has 
$$\Theta^i \cdot f^\ast {\Theta'}^{n-i} > 0$$ for $0 \leq i \leq n$ (see \cite{CAV}, Lemma 4.3.2, p. 76).
As a consequence, if we drop from the right hand side all terms except the two with $i = n, n-1$, and then divide by $n!$, we obtain the inequality 
$$\chi(L|_S) = \frac{(L|_S)^n}{n!} \geq a\; \dfrac{\Theta^n}{n!} + \dfrac{1}{a}\; \dfrac{\Theta^{n-1}}{(n-1)!} \cdot f^\ast \Theta'.$$

\subsection{On subvarieties of bounded degree $-$ I}
\label{nec}

Assume now that $A$ is a simple Abelian variety.
We study the set
$$\AS_{A \times A' / \,p} := 
\Big\{ S \in \AS_{A \times A'}
\mbox{ s.t. } p{}_S : S \rightarrow A \mbox{ is an isogeny}
\Big\}$$
and give an asymptotic estimate for the counting function
$$N_{A \times A' / \,p}(t) := 
\# \Big\{ S \in \AS_{A \times A' / \,p} \mbox{ s.t. } \chi(L|_S) \leq t \Big\}.$$
The estimate is going to be in terms of the quantity 
$$H := \mbox{the rank of the group } Hom(A,A').$$

\begin{theorem}
\label{AA'p}
In the present setting, with $A$ simple, one has 
$$N_{A \times A' / \,p}(t) = O\left(t^{H+1}\right).$$
\end{theorem}

\begin{proof}
We apply the description in \S \ref{parametrization1}.
The set $\AS_{A \times A' / \,p}$ is bijective to the set of parameter data 
$$\Big\{ (a,f) \in \mathbb Z_{>0} \times Hom(A, A') \mbox{ s.t. } \deg(a,f) = a^{2n-1} \Big\}.$$ 
Thus $N_{A \times A' / \,p}(t)$ is the number of elements of the set 
$$\Big\{ (a,f) \in \mathbb Z_{>0} \times Hom(A, A') \ \mbox{s.t.} \ \deg(a,f) = a^{2n-1} \ \mbox{and} \ 
\dfrac{\deg C(a,f)}{n!} \leq t \Big\}.$$
Then we consider the description in \S \ref{degree1}.
We have the inequality
$$\frac{\deg C_{a,f}}{n!} \geq a\; \dfrac{\Theta^n}{n!} + \dfrac{1}{a}\; \dfrac{\Theta^{n-1}}{(n-1)!} \cdot f^\ast \Theta'.$$
Hence the codnition $\frac{\deg C_{a,f}}{n!} \leq t$ implies 
$$a \leq \dfrac{t}{M} \ \mbox{ and } \
\dfrac{\Theta^{n-1}}{(n-1)!} \cdot f^\ast \Theta' \leq a\, t,$$ 
where $M := \Theta^n / n!$.
It follows that the set of parameter data above is a subset of the product
$$\Big\{ a \in \mathbb Z \mbox{ s.t. } 1 \leq a \leq \frac{t}{M} \Big\} \times 
\Big\{ f \in Hom(A,A') \ \mbox{s.t.} \ \dfrac{\Theta^{n-1}}{(n-1)!} \cdot f^\ast \Theta' \leq \frac{t^2}{M}\Big\}.$$
Since from Proposition \ref{ellipsoid} we know that the cardinality of the set on the right hand side of the product is a function of order $O(t^{H})$, 
it follows that $N_{A \times A' / \,p}(t)$ 
is bounded above by a function of order $O(t)\, O(t^{H})$
as in the statement.
\end{proof}

\section{Subvarieties of a product variety}

Keeping the notation introduced in the preceding section, we extend the description of subvarieties in the product variety $A \times A'$.

\subsection{Parametrization}
\label{parametrization2}
We describe the Abelian subvarieties $S$ of $A \times A'$ such that the projection
$p{}_S: S \rightarrow A$ is surjective. 

Define $F:= Ker_0(p_S)$ and view it as an Abelian subvariety of $A'$.
Let $T$ in $S$ be an Abelian subvariety such that
$F + T = S$ and $F \cap T$ is finite. Then the sum map
$s: F \times T \rightarrow S$ is an isogeny, and
the restricted projection $p_T: T \rightarrow A$
is an isogeny too. 

Assume that a polarization $L$ is given on $A \times A'$. 
Because of the reducibility theorem relative to a given polarization (cfr. \S \ref{RP}),
we can choose $T$ in such a way that
$s^\ast (L|_S) = {q}^\ast (L|_F) + {q'}{}^\ast (L|_{T})$
where $q$ and $q'$ denote the projections 
from $F \times T$ to $F$ and $T$, respectively.

We say that two subvarieties $F$ and $T$ in $A \times A'$ are {complementary} with respect to $L$ when the sum map $s: F \times T \rightarrow F + T$ is an isogeny and the pullback polarization $s^\ast (L|_{F + T})$ splits as ${q}^\ast (L|_F) + {q'}^\ast (L|_{T})$.

There is a correspondence 
$$\left\{ \mbox{\begin{tabular}{c} 
$(F,T) \in \AS_{A'} \times \AS_{A \times A' / \, p}$ \\  
s.t. complementary w.r.t. $L$ 
\end{tabular}} \right\} \longrightarrow
\left\{ \mbox{\begin{tabular}{c} 
$S \in \AS_{A \times A'}$  s.t.  \\  
$p_S: S \rightarrow A$ surjective
\end{tabular}} \right\}$$
$$\hspace{38pt} (F,T) \hspace{10pt} \longmapsto \hspace{10pt} F + T $$
and this is a bijective correspondence. 
That the correspondence is surjective is clear from the construction.
Given $S$, the choice of $F$ is unique, because $Ker(p_S)$ 
must be equal to $(T \cap F) + F$ and therefore $Ker_0(p_S) = F$.
That the choice of $T$ is then unique is part of the reducibility theorem.

\subsection{Degree}
\label{degree2}
Concerning the (reduced) degree of $S = F + T$ with respect to the polarization $L = p^\ast\Theta + {p'}{}^\ast\Theta'$, i.e. the Euler characteristic $\chi(L|_S)$, we here recall what we know. Define 
$$a := \deg(s: F \times T \rightarrow S) = \# (F \cap T) = \deg(p{}_T: T \rightarrow A).$$
From \S \ref{remark} we have the basic facts
$$a\, \chi(L|_S) = \chi(\Theta'|_F)\, \chi(L|_T) \ \ \ \mbox{ and } \ \ \ 
\chi(L|_T) \leq \chi(L|_S)\, \chi(\Theta'|_F) \hspace{20pt}$$
and from \S \ref{degree1} we take the inequality
$$a M \leq \chi(L|_T)$$
where $M = \chi(\Theta)$, holding because $T$ is of the form $C_{a,f}$ with $a = \deg(p{}_T)$ as above.

\subsection{On subvarieties of bounded degree $-$ II}
Assume now that $A$ is a simple Abelian variety.
\begin{theorem}
\label{AA'}
If $A$ is simple and if $N_{A'}(t) = O(t^E)$ then $N_{A \times A'}(t) = O\left(t^{E+2(H+1)}\right)$.
\end{theorem}
\begin{proof}
Since $A$ is simple, the set $\AS_{A \times A'}$ is the disjoint union of $\AS_{A'}$
and the complementary subset which, according to \S \ref{parametrization2}, 
is bijective to the set of parameter data 
$$\Big\{ (F,T) \in \AS_{A'} \times \AS_{A / \, p}
\mbox{ \ s.t. $F$ and $T$ are complementary w.r.t. $L$} \Big\}.$$

The basic facts concerning the degree of $S = F + T$ expressed in terms of the degree $a$ of the sum isogeny $F \times T \rightarrow S$ and the degrees of $F$ and $T$ are collected in \S \ref{degree2}. In the present setting we can show that
$$\chi(L|_S) \leq t \ \ \ \mbox{ implies } \ \ \ 
\chi(\Theta'|_F) \leq t \ \ \mbox{ and } \ \ \chi(L|_T) \leq t^2.$$
We have $a\, \chi(L|_S) = \chi(\Theta'|_F) \chi(L|_T)$. So from $\chi(L|_T) \geq a M$  we obtain $\chi(\Theta'|_F) \leq t/M$ and from $\chi(L|_T) \leq \chi(L|_S) \chi(\Theta'|_F)$ we obtain $\chi(L|_T) \leq t^2/M$.
Note that in the weaker inequalities stated above we dropped the constant $M$, just because it will not affect the exponent in the estimate of the counting function $N_{A \times A'}(t)$.

Summing everything up, we have
$$N_{A \times A'}(t) \leq N_{A'}(t) +
N_{A'}(t) \, {N}_{A \times A' / \, p}(t^2).$$
From Theorem \ref{AA'p} we know that
$${N}_{A \times A'/ \, p}(t^2) = O(t^{2(H+1)}).$$
So, with the hypothesis that $N_{A'}(t) = O(t^E)$, the statement follows.
\end{proof}

\section{Subvarieties of a completely reduced product}
\label{finale}

Let $A = B^k$ be the $k$th self product of a simple Abelian variety $B$, of dimension $n$, endowed with a split polarization $L = p_1^\ast\Theta_1 + \cdots + p_k^\ast\Theta_k$, where each $\Theta_i$ is a polarization on $B$. 
The estimate for the counting function $N_{B^k}(t)$ is in terms of  
$h$, the rank of the endomorphism group $End(B)$.

\begin{theorem}
\label{B^k}
There is an asymptotic estimate
$$N_{B^k}(t) = O\left(t^{(kh+2)(k-1)}\right).$$
\end{theorem}

\begin{proof}
By induction on $k$. The initial case $k=1$ being obvious, assume $k \geq 2$.
 Write $B^k = B \times B^{k-1}$ and let $p: B^k \rightarrow B$ and $p': B^k \rightarrow B^{k-1}$ be the natural projections, with $p := p_1$ and $p' := (p_2,\ldots,p_k)$. The given polarization on $B^k$ splits as $L = p^\ast \Theta + p'{}^\ast \Theta'$ where $\Theta := \Theta_1$ and where $\Theta'$ is the split polarization on $B^{k-1}$ obtained from $\Theta_2,\ldots,\Theta_k$. 

The group $Hom(B,B^{k-1})$ is of rank $H = (k-1)h$.
By the induction hypothesis we have $N_{B^{k-1}}(t) = O(t^{E})$
with $E = \left((k-1)h+2\right)(k-2)$.
From Theorem \ref{AA'}
we have $N_{B^k}(t) = (t^{E + 2(H+1)})$. In the present setting 
$E + 2(H+1) = (kh+2)(k-1)$, and the statement is proven.
\end{proof}

Let $A = B_1^{k_1} \times \cdots \times B_q^{k_q}$ be a product of simple Abelian varieties, each occurring with some multiplicity, and assume that $B_i$ and $B_j$ are not isogenous for $i \neq j$.

Each simple Abelian subvariety of $A$ is contained in a unique $B_j^{k_j}$ and is then isogenous to $B_j$.
An arbitrary Abelian subvariety of $A$ may be written as $S = S_1 \times \cdots \times S_q$ where 

\begin{tabular}{ccl}
$S_j$ & = & the sum of all simple Abelian subvarieties of $S$ contained in $B_j^{k_j}$ \\ & & (and hence belonging to the isogeny class of $B_j$) \smallskip \\ 
& = & $S \cap B_j^{k_j}\, .$ \smallskip
\end{tabular} \\
Thus we have a correspondence 
$$\AS_{B_1^{k_1}} \times \cdots \times \AS_{B_q^{k_q}} \longrightarrow \AS_{A}$$
$$(S_1,\ldots,S_q) \longmapsto S_1 \times \ldots \times S_q$$
and clearly this is a bijective correspondence.

Let $L$ be a split polarization on $A$. It may be written as $L = p_1^\ast(\Theta_1) + \cdots + p_q^\ast(\Theta_q)$ where each $\Theta_j$ is a split polarization on $B_j^{k_j}$. If $S = S_1 \times \ldots \times S_q$ is an Abelian subvariety of $A$, as above, then 
(the K\"unneth formula again)
$$\chi(L|_S) = \chi(\Theta_1|_{S_1}) \cdots \chi(\Theta_q|_{S_q}).$$

In the present setting, the estimate for the function $N_A(t)$, and its proof, involve the following data: 

\begin{center}
\begin{tabular}{cl}
$h_j$ & the rank of the endomorphism group $End(B_j)$, \smallskip \\
$h$ & the maximum among the ranks $h_1,\ldots,h_q$, \smallskip \\
$k$ & the maximum among the multiplicities $k_1,\ldots,k_q$. \smallskip
\end{tabular} 
\end{center}

\begin{theorem}
\label{B1Bq}
There is an asymptotic estimate
$$N_{B_1^{k_1} \times \cdots \times B_q^{k_q}}(t) = O \left(t^{\,q (kh+2)(k-1)} \right).$$
\end{theorem}

\begin{proof}
We have seen that
the set $\AS_{A}$ is bijective to $\AS_{B_1^{k_1}} \times \cdots \times \AS_{B_q^{k_q}}$ and from the expression for $\chi(L|_S)$ above it follows that
$N_A(t) \leq N_{B_1^{k_1}}(t) \cdots N_{B_q^{k_q}}(t)$.
From Theorem \ref{B^k} we have 
$N_{B_j^{k_j}}(t) = O\left(t^{(k_jh_j+2)(k_j-1)}\right)$ and therefore
$N_{B_1^{k_1}}(t) \cdots N_{B_q^{k_q}}(t) = O \left(t^{(k_1h_1+2)(k_1-1) + \cdots + (k_qh_q+2)(k_q-1)} \right)$,
and so the statement follows.
\end{proof}

According to Remark \ref{R} the proof of Theorem \ref{T} is now complete.

\vfill \noindent
email: {\tt lucioguerra56@gmail.com}

\end{document}